\documentclass[11pt,a4paper,twoside,leqno]{amsart}

\usepackage{amsmath,amssymb,amsthm,mathtools}
\usepackage{eucal}
\usepackage{mathrsfs}
\usepackage[english]{babel}
\usepackage{fullpage}
\usepackage[dvipsnames]{xcolor}
\definecolor{britishracinggreen}{rgb}{0.0,0.26,0.15}
\definecolor{cobalt}{rgb}{0.0,0.28,0.67}
\usepackage[
  colorlinks=true,
  bookmarks=true,
  citecolor=britishracinggreen,
  linkcolor=cobalt,
  filecolor=black,
  urlcolor=black
]{hyperref}

\numberwithin{equation}{section}
\newcommand{\A}{\mathbb{A}}
\newcommand{\C}{\mathbb{C}}

\newcommand{\Z}{\mathbb{Z}}

\newcommand{\SH}{\mathrm{SH}}
\newcommand{\SHext}{\mathrm{SH}_{\mathrm{ext}}}
\newcommand{\Ho}{\mathrm{Ho}}
\newcommand{\CAlg}{\operatorname{CAlg}}
\newcommand{\K}{\mathrm{K}}
\newcommand{\KH}{\mathrm{KH}}
\newcommand{\KGL}{\mathrm{KGL}}

\newcommand{\op}{\mathrm{op}}
\newcommand{\one}{\mathbf{1}}
\newcommand{\tensor}{\otimes}

\newcommand{\mapSH}[3]{\operatorname{map}_{\SH(#1)}(#2,#3)}
\newcommand{\colim}{\operatorname*{colim}}
\newcommand{\Spec}{\operatorname{Spec}}
\newcommand{\Th}{\operatorname{Th}}

\newcommand{\Hom}{\operatorname{Hom}}
\newcommand{\uHom}{\underline{\Hom}}

\theoremstyle{plain}
\newtheorem{thm}{Theorem}[section]
\newtheorem{alphthm}{Theorem}

\newtheorem{prop}[thm]{Proposition}
\newtheorem{lem}[thm]{Lemma}
\newtheorem{cor}[thm]{Corollary}
\newtheorem*{lem*}{Lemma}

\theoremstyle{definition}
\newtheorem{defn}[thm]{Definition}
\newtheorem{setup}[thm]{Setup}

\theoremstyle{remark}

\newtheorem*{rem*}{Remark}

\title[Stable motivic invariants]{Motivic invariants of birational hyper-K\"ahler manifolds}
\author[W.~Wu]{Weimufei Wu}
\address{Northwestern University}
\email{weimufeiwu2030@u.northwestern.edu}
\date{}

\begin{document}
\vspace{-3em}

\begin{abstract}
We prove that birational projective irreducible holomorphic symplectic manifolds have equivalent integral cohomological motives as commutative algebra objects in $\SH(\C)$. Consequently, their algebraic K-theory spectra are equivalent as $E_\infty$-algebras, and their integral graded K-theory rings are isomorphic.
\end{abstract}

\maketitle

\section{Introduction}\label{sec:introduction}

\subsection*{Main result}
A folklore theorem asserts that birational hyper-Kähler manifolds have isomorphic motivic invariants, compatibly with their ring structures. In this note, we establish the following formulation of this statement.

\begin{alphthm}\label{thm:HK-intro}
Let $X$ and $X'$ be birational projective irreducible holomorphic symplectic manifolds, and let $p$ and $p'$ be their structure maps. Then there is an equivalence
\[
a\colon p'_*\one_{X'}\xrightarrow{\simeq}p_*\one_X
\]
in $\CAlg(\SH(\C))$. In particular, $a$ induces an equivalence
\[
\K(X')\xrightarrow{\simeq}\K(X)
\]
of the algebraic K-theory spectra as $E_\infty$-algebras and an isomorphism of integral graded $\K_*$-rings in all degrees.
\end{alphthm}

For birational projective irreducible holomorphic symplectic manifolds, \cite[Theorem~3.2]{Rie14} gives an integral graded Chow-ring isomorphism, and \cite[Corollary~7.3]{Rie14} gives an isomorphism of rational Grothendieck $\K_0$-rings. Her argument also gives an isomorphism of rational Chow motives as algebra objects; see \cite[Theorem~4.16(iv)]{FV20}. Theorem~\ref{thm:HK-intro} provides an integral comparison at the level of the stable motivic homotopy $\infty$-category as commutative algebra objects, and hence recovers multiplicative structures on represented cohomology theories.

For varieties of K3$^{[n]}$-type, the D-equivalence theorem of \cite[Theorem~1.2]{MSYZ25} also yields an equivalence of algebraic K-theory spectra, by Morita invariance \cite[Theorems~1.3 and~1.10]{BGT13}. The distinction here is multiplicative: a Fourier--Mukai equivalence need not preserve tensor products and therefore does not, by itself, yield the graded-ring comparison in Theorem~\ref{thm:HK-intro}.

\subsection*{The geometric input}

The geometric input in \cite{Rie14} consists of two algebraic spaces over one smooth pointed curve, with isomorphic restrictions to the punctured curve. Our proof starts with the same geometric input, and we name it as an \textit{H-relation} as follows; see Definition~\ref{defn:H-relation}.

An H-relation consists of two smooth proper families over the same smooth pointed curve, identifications of the special fibers with $X$ and $X'$, and an isomorphism over the punctured curve; see Definition~\ref{defn:H-relation}. Section~\ref{sec:riess} verifies that Rie{\ss}'s deformation families supply such H-data after the curve is shrunk around the marked point.

Our approach starts with an H-relation, and combines the theory of motivic six functors for algebraic spaces and motivic nearby cycle functors for schemes, to get the following.

\begin{alphthm}[Specialization over a curve]\label{thm:H-invariance-intro}
Let $k$ be a field. Let $X$ and $X'$ be smooth proper $k$-schemes that are related by an H-relation in the sense of Definition~\ref{defn:H-relation}, and let $p$ and $p'$ be their structure maps. Then there is an equivalence
\[
a_{\gamma}\colon p'_*\one_{X'}\xrightarrow{\simeq}p_*\one_X
\]
in $\CAlg(\SH(k))$. By duality, there is an equivalence
\[
h_{\gamma}\colon \Sigma_T^\infty X_+ \xrightarrow{\simeq} \Sigma_T^\infty X'_+
\]
of counital cocommutative coalgebra objects in $\SH(k)$, with the structures given by the geometric diagonal and augmentation.

\end{alphthm}

\subsection*{Summary of the proof}

For a smooth proper family over a pointed curve $q\colon Z\to (C,0)$, put
\[
A=q_*\one_Z\in\CAlg(\SH(C)).
\]
The six-functor formalism for algebraic spaces in \cite{Cho24} makes the underlying object of $A$ dualizable.

The isomorphism over the punctured curve $U$ in an H-relation gives an equivalence between commutative algebras over $U$. The unipotent nearby cycle functor, in the sense of \cite{Ayo07b} as used in \cite{JY22}, specializes this algebra equivalence to the special fibers; duality then turns the algebra comparison into the desired equivalence of counital cocommutative coalgebras. Sections~\ref{sec:duality}--\ref{sec:specialization} carry out the cohomological consequences of an H-relation, and Section~\ref{sec:riess} verifies that Rie{\ss}'s construction gives an H-relation and proves Theorem~\ref{thm:HK-intro}.

\subsection*{Acknowledgements}
I would like to thank my advisor, Yuchen Liu, for his constant support and encouragement, and for many inspiring conversations. I am also grateful to Benjamin Antieau, Dinglong Wang, Jakub Witaszek, Zheng Xu, Qizheng Yin, Bogdan Zavyalov, and Zhi Zheng for helpful discussions and for answering my questions, and to Junliang Shen and Brian Shin for stimulating conversations following their inspiring talks. I also thank Lie Fu for reading a preliminary version of this manuscript and for his many helpful suggestions, particularly concerning the formulation of the main theorem.

Special thanks go to my undergraduate friends, with whom I spent many happy hours learning the six-functor formalism of D-modules and stable homotopy theory.

\subsection*{AI disclosure}

This work began in May 2026 with the author's observation that Rie\ss's construction could be combined with Ayoub's motivic nearby cycle functor and Chowdhury's six-functor formalism for algebraic spaces to establish the main result of this paper. A key obstacle was the lack of a suitable motivic nearby-cycle formalism for general algebraic spaces. In October 2026, after being given a conditional proof, GPT-6 Pro suggested using a result of Jin and Yang on dualizable objects to circumvent this obstacle and obtain an unconditional proof; see Proposition~4.1. AI assistance was also used to draft and check the manuscript and to locate references. All AI-assisted material was independently verified and revised by the author, who takes full responsibility for the contents of the paper.

\section{Preliminaries}\label{sec:preliminaries}

All base curves are smooth, connected, separated, and of finite type over the indicated field. 

If $B$ is a scheme, $\SH(B)$ denotes the stable motivic homotopy $\infty$-category over $B$, and $\Ho(\SH(B))$ denotes its homotopy category. All algebraic spaces in this paper are finite-dimensional, noetherian, quasi-compact, and quasi-separated. For such an algebraic space $Y$, write $\SHext(Y)$ for Chowdhury's extended stable $\infty$-category. These algebraic spaces belong to the Nisnevich-local-section class of \cite[Example~2.1.6(1) and Theorem~2.2.2]{Cho24}, so they lie in the scope of that extension; on schemes, $\SHext$ agrees with $\SH$. The extension supplies the six operations, proper base change, projection formula, and smooth purity used below \cite[Theorem~4.5.1]{Cho24}. 

For a symmetric monoidal $\infty$-category $\mathcal D$, write $\CAlg(\mathcal D)$ for its $\infty$-category of commutative algebra objects. The right adjoint of a symmetric monoidal functor has its canonical lax symmetric monoidal structure; we use the induced adjunction on commutative algebras \cite[Corollary~7.3.2.7]{Lur17}.

For $M,E\in\SH(k)$, the notation $\mapSH{k}{M}{E}$ means the ordinary mapping spectrum of $\SH(k)$. For a scheme or algebraic space $B$ in the standing class, write
\[
D_BM=\uHom_B(M,\one_B)
\]
for internal Hom into the unit, whenever the ambient category is clear. For a smooth $k$-scheme $X$, put
\[
M(X)=\Sigma_T^\infty X_+,\qquad T=\A^1/(\A^1\setminus\{0\}).
\]
The geometric diagonal and structural map equip $M(X)$ with its natural counital cocommutative coalgebra structure.
For $E\in\SH(k)$, we write
\[
E^{a,b}(X)=[M(X),\Sigma^{a,b}E]_{\Ho(\SH(k))},
\qquad
E^{*,*}(k)=E^{*,*}(\Spec(k)).
\]
For a smooth morphism $f$ in this scheme or algebraic-space setting, we write $f_\#$ for the left adjoint to $f^*$ supplied by \cite[Proposition~4.1.2]{Cho24}.

\begin{defn}[The H-relation]\label{defn:H-relation}
Let $k$ be a field, and let $X$ and $X'$ be smooth proper $k$-schemes. An H-relation from $X$ to $X'$ consists of the following data:
\begin{enumerate}
\item a curve $C$ over $k$;
\item a $k$-point $0\in C$, with $U=C\setminus\{0\}$ and inclusions $i\colon\{0\}\to C$ and $j\colon U\to C$;
\item morphisms
\[
q\colon Z\to C,\qquad q'\colon Z'\to C
\]
from algebraic spaces satisfying the standing conventions above, each smooth, proper, and of finite type over $C$;
\item identifications of the special fibers $Z_0$ and $Z'_0$ with $X$ and $X'$ as $k$-schemes;
\item a $C$-isomorphism
\[
\gamma\colon Z_U\xrightarrow{\simeq}Z'_U.
\]
\end{enumerate}
\end{defn}
The letter $H$ alludes both to the non-Hausdorff behavior of moduli spaces and to Huybrechts's deformation theorem for birational hyper-K\"ahler manifolds.

\begin{lem}[A coordinate at the marked point]\label{lem:etale-coordinate}
Let $k$ be a field, let $C$ be a smooth curve over $k$, and let $0\in C(k)$. After replacing $C$ by a Zariski neighborhood of $0$, there is a morphism
\[
t\colon C\to \A^1_k
\]
which is \'etale and whose scheme-theoretic zero fiber is
\[
C\times_{\A^1_k}\{0\}\simeq \Spec(k).
\]
\end{lem}

\begin{proof}
Choose an affine neighborhood $C_1$ of $0$ and a regular function $t$ whose germ is a uniformizer in $\mathcal O_{C,0}$. Since $C$ is smooth of dimension one at $0$, the differential $dt$ is nonzero at $0$, so $t\colon C_1\to\A^1_k$ is \'etale on an open neighborhood of $0$. The zero locus of $t$ is closed and contains $0$ with multiplicity one. Removing its other points and the non-\'etale locus gives a Zariski neighborhood $C\subset C_1$ on which $t$ is \'etale; the remaining zero fiber is reduced and has the single $k$-rational point $0$.
\end{proof}

\begin{setup}\label{setup:nearby}
For the rest of the paper, when H-data over $C$ are given, we shrink $C$ around $0$ and choose $t\colon C\to\A^1_k$ as in Lemma~\ref{lem:etale-coordinate}. We keep the notation
\[
U=C\setminus\{0\}=t^{-1}(\mathbb G_{m,k})
\]
and write
\[
\Upsilon_t\colon \SH(U)\to \SH(k)
\]
for Ayoub's unipotent motivic nearby-cycle functor associated with $t$ \cite[Section~4]{Ayo07c}, in the form recalled in \cite[Section~2.2.5, formula~(2.2.5.1)]{JY22}. We use its lax symmetric monoidal structure from \cite[Section~2.2.10]{JY22}.
\end{setup}

\section{Dualizability}\label{sec:duality}

For a smooth proper morphism between schemes $M\to C$, the proper pushforward of the unit in $\SH(M)$ is dualizable in $\SH(C)$. We recall the same argument for algebraic spaces.

\begin{prop}[Dualizability of the proper pushforward]\label{prop:dualizable-pushforward}
Let $C$ be a smooth connected curve over a field, and let
\[
q\colon Z\to C
\]
be a morphism from an algebraic space satisfying the standing conventions of Section~\ref{sec:preliminaries}. Assume that $q$ is smooth, proper, and of finite type. Put
\[
A=q_*\one_Z\in\SH(C),
\qquad
L=q^!\one_C\simeq \Th(T_q)\in\SHext(Z).
\]
Then $A$ is strongly dualizable in $\SH(C)$, and its dual is
\[
D_C A\simeq q_*L\simeq q_\#\one_Z.
\]
For every $E\in\SH(C)$, the canonical map
\[
D_C A\tensor E\longrightarrow \uHom(A,E)
\]
is an equivalence.
\end{prop}

\begin{proof}
We work in Chowdhury's extension for the algebraic space $Z$ and in $\SH(C)$ on the base. Properness gives $q_!\simeq q_*$ by \cite[Proposition~4.3.1]{Cho24}, smooth full-Thom purity gives $q^!\one_C\simeq \Th(T_q)$ by \cite[Proposition~4.4.1 and its proof, and Remark~4.4.2]{Cho24}, and the smooth adjunction and projection formula are supplied by \cite[Proposition~4.1.2 and Theorem~4.5.1]{Cho24}. For $E\in\SH(C)$ these inputs give the canonical internal-Hom chain
\[
\begin{aligned}
\uHom_C(q_*\one_Z,E)
&\simeq \uHom_C(q_!\one_Z,E)\\
&\simeq q_*q^!E\\
&\simeq q_*\bigl(\Th(T_q)\tensor q^*E\bigr)\\
&\simeq q_*\Th(T_q)\tensor E.
\end{aligned}
\]
Taking $E=\one_C$ identifies $D_CA$ with $q_*\Th(T_q)=q_*L$. The same canonical chain for general $E$ identifies the strong-duality comparison $D_CA\tensor E\to\uHom(A,E)$, so $A$ is strongly dualizable, as in the formal argument of \cite[Section~3, equations (3-1)--(3-4), p.~3616]{Hoy14}. Finally, purity identifies $q_\#\one_Z$ with $q_!q^!\one_C$, and properness identifies this object with $q_*L$.
\end{proof}

\section{Specialization}\label{sec:specialization}

Ayoub's nearby-cycle formalism supplies the specialization map and the tensor compatibilities used below. We work in the $\infty$-categories of commutative algebras.

Let
\[
\chi_t=i^*j_*\colon \SH(U)\to\SH(k).
\]
There is a canonical comparison $\rho_t\colon\chi_t\to\Upsilon_t$. For $M\in\SH(C)$, define
\[
\operatorname{sp}^t_M\colon i^*M\longrightarrow\Upsilon_tj^*M
\]
as the composite
\[
i^*M\longrightarrow i^*j_*j^*M=\chi_t(j^*M)
\xrightarrow{\rho_t(j^*M)}
\Upsilon_t(j^*M),
\]
where the first map is $i^*$ applied to the unit of $j^*\dashv j_*$.

\begin{prop}[Specialization for commutative algebras]\label{prop:specialization-products}
Let $C$, $0$, $U$, $i$, $j$, and $t$ be as in Setup~\ref{setup:nearby}. The functor $\Upsilon_t$ admits a refinement
\[
\Upsilon_t\colon\CAlg(\SH(U))\longrightarrow\CAlg(\SH(k)).
\]
For $A\in\CAlg(\SH(C))$, there is a canonical specialization map
\[
\operatorname{sp}^t_A\colon i^*A\longrightarrow\Upsilon_tj^*A
\]
in $\CAlg(\SH(k))$, naturally in $A$. If the underlying object of $A$ is dualizable in $\SH(C)$, then this map is an equivalence.
\end{prop}

\begin{proof}
Let $B^\bullet$ be the cosimplicial smooth $\mathbb G_{m,k}$-scheme of \cite[Section~2.2.3]{JY22}, and write
\[
b_n\colon B^n\times_{\mathbb G_{m,k}}U\longrightarrow U,
\qquad t_U=t|_U.
\]
For $E\in\SH(U)$, smooth adjunction and the projection formula give
\[
\uHom_U(t_U^*\Sigma_T^\infty B^n_+,E)
\simeq (b_n)_*b_n^*E.
\]
Here the suspension spectrum is taken over $\mathbb G_{m,k}$. If $E$ is a commutative algebra, the objects $i^*j_*(b_n)_*b_n^*E$ form a simplicial diagram in $\CAlg(\SH(k))$: pullbacks are symmetric monoidal, their right adjoints are lax symmetric monoidal, and the transition maps are induced by the adjunction units. Its geometric realization
\[
\colim_{[n]\in\Delta^{\op}}^{\CAlg(\SH(k))}
 i^*j_*(b_n)_*b_n^*E
\]
has underlying object $\Upsilon_t(E)$ by \cite[formula~(2.2.5.1)]{JY22} and \cite[Corollary~3.2.3.2]{Lur17}. This gives the stated refinement of $\Upsilon_t$.

For $A\in\CAlg(\SH(C))$, the adjunction units
\[
A\longrightarrow (jb_n)_*(jb_n)^*A
\]
are compatible morphisms of commutative algebras. Pulling back by $i^*$ and passing to the geometric realization therefore gives
\[
\operatorname{sp}^t_A\colon i^*A\longrightarrow\Upsilon_tj^*A
\]
in $\CAlg(\SH(k))$, naturally in $A$. Its underlying map is the canonical specialization map defined above.

The morphism $t$ is \'etale, hence smooth. Every dualizable object of $\SH(C)$ is universally strongly locally acyclic over $\A^1_k$ in the sense of \cite[Definition~3.1.2 and Section~3.1.3, first bullet]{JY22}. Therefore \cite[Proposition~3.1.7(1)]{JY22} identifies the underlying specialization map as an equivalence. Since equivalences of commutative algebras are detected on their underlying objects \cite[Lemma~3.2.2.6]{Lur17}, the result follows.
\end{proof}

The H-relation between two schemes $X$ and $X'$ now gives a comparison of their motives as commutative algebras.

\begin{prop}[Algebra comparison on special fibers]\label{prop:algebra-comparison}
Let $k$ be a field, and let
\[
(q\colon Z\to C,\ q'\colon Z'\to C,\ \gamma\colon Z_U\xrightarrow{\simeq}Z'_U)
\]
be an H-relation from $X$ to $X'$ as in Definition~\ref{defn:H-relation}. Let
\[
p\colon X\to \Spec(k),\qquad p'\colon X'\to\Spec(k)
\]
be the structure maps, and put
\[
B_X=p_*\one_X,\qquad B_{X'}=p'_*\one_{X'}.
\]
Equip $B_X$ and $B_{X'}$ with the commutative algebra structures supplied by the lax symmetric monoidal structures on $p_*$ and $p'_*$. Then there is an equivalence
\[
a_{\gamma}\colon B_{X'}\xrightarrow{\simeq}B_X
\]
in $\CAlg(\SH(k))$.
\end{prop}

\begin{proof}
Put $A=q_*\one_Z$ and $A'=q'_*\one_{Z'}$, with their commutative algebra structures supplied by the lax symmetric monoidal structures on $q_*$ and $q'_*$. Let
$\varepsilon_q\colon q^*A\to\one_Z$ be the adjunction counit.
The multiplication $\mu_A\colon A\tensor A\to A$ is adjoint to
\[
q^*(A\tensor A)
\simeq q^*A\tensor q^*A
\xrightarrow{\varepsilon_q\tensor\varepsilon_q}
\one_Z\tensor\one_Z
\simeq\one_Z,
\]
and the unit $\eta_A\colon\one_C\to A$ is adjoint to
$q^*\one_C\simeq\one_Z$. The same description applies to $A'$ and to the specified algebra structures on $B_X$ and $B_{X'}$.

Proper base change for algebraic spaces
\cite[Proposition~4.1.5]{Cho24} gives canonical equivalences
\[
\beta_X\colon i^*A\xrightarrow{\simeq}B_X,
\qquad
j^*A\xrightarrow{\simeq}(q_U)_*\one_{Z_U},
\]
and similarly for $A'$.
These are equivalences of commutative algebras.
Indeed, writing $\iota\colon X\to Z$ for the special-fiber inclusion,
the counit gives a morphism
\[
p^*i^*A\simeq\iota^*q^*A
\xrightarrow{\iota^*\varepsilon_q}\one_X
\]
in $\CAlg(\SH(X))$. Its adjoint under the induced adjunction on commutative algebras is $\beta_X$. The underlying map is the proper-base-change equivalence, so $\beta_X$ is an equivalence in $\CAlg(\SH(k))$. The same argument applies to $A'$ and to the restrictions over $U$.

The isomorphism $\gamma\colon Z_U\xrightarrow{\simeq}Z'_U$ therefore
induces an equivalence
\[
\alpha_\gamma\colon j^*A'\xrightarrow{\simeq}j^*A
\]
in $\CAlg(\SH(U))$, using the monoidal pullback along $\gamma$ and the induced adjunctions on commutative algebras.

By Proposition~\ref{prop:dualizable-pushforward}, the underlying objects of $A$ and $A'$ are dualizable. Proposition~\ref{prop:specialization-products} therefore gives equivalences
\[
s_A=\operatorname{sp}^t_A\colon i^*A\xrightarrow{\simeq}\Upsilon_tj^*A,
\qquad
s_{A'}=\operatorname{sp}^t_{A'}\colon i^*A'\xrightarrow{\simeq}\Upsilon_tj^*A'
\]
in $\CAlg(\SH(k))$. Applying the commutative-algebra refinement of $\Upsilon_t$ to $\alpha_\gamma$ and composing gives
\[
a_\gamma
=
\beta_X\circ s_A^{-1}\circ\Upsilon_t(\alpha_\gamma)
\circ s_{A'}\circ\beta_{X'}^{-1}
\colon B_{X'}\xrightarrow{\simeq}B_X
\]
in $\CAlg(\SH(k))$.
\end{proof}

\begin{proof}[Proof of Theorem~\ref{thm:H-invariance-intro}]
The algebra equivalence is Proposition~\ref{prop:algebra-comparison}. Let $M(X)=\Sigma_T^\infty X_+=p_\#\one_X$ and $M(X')=\Sigma_T^\infty X'_+=p'_\#\one_{X'}$. Hoyois's smooth-proper duality gives strong dualizability and the natural identification
\[
\uHom_k(p_\#\one_X,E)\simeq p_*p^*E
\]
by \cite[Corollary~6.13 and its proof, and Remark~6.21]{Hoy17}; taking $E=\one_k$ identifies $D_kM(X)$ with $p_*\one_X$, and similarly for $X'$. This identification is compatible with the lax symmetric monoidal structures supplied by the monoidal adjunctions, so the commutative algebra structure on $p_*\one_X$ is dual to the natural counital cocommutative coalgebra structure on $M(X)$. In particular, the adjunction multiplication is the dual of the geometric diagonal, and the adjunction unit $\one_k\to p_*\one_X$ is the dual of the augmentation of $M(X)$ given by $p\colon X\to\Spec(k)$. The same applies to $X'$.

Proposition~\ref{prop:algebra-comparison} gives an equivalence
\[
a_\gamma\colon D_kM(X')\xrightarrow{\simeq}D_kM(X)
\]
in $\CAlg(\SH(k))$. Duality on dualizable objects is contravariantly symmetric monoidal. Dualizing therefore gives an equivalence
\[
h_\gamma=D_k(a_\gamma)\colon M(X)\xrightarrow{\simeq}M(X')
\]
of counital cocommutative coalgebra objects in $\SH(k)$.
\end{proof}

\begin{cor}[Specialization from an isomorphic generic fiber]\label{cor:generic-fiber-specialization}
Let $k$ be a field, let $C$ be a smooth connected separated finite-type curve over $k$ with $0\in C(k)$, and put $K=k(C)$. Let
\[
q\colon Z\to C,\qquad q'\colon Z'\to C
\]
be morphisms from algebraic spaces satisfying the standing conventions of Section~\ref{sec:preliminaries}, each smooth, proper, and of finite type, whose special fibers $X=Z_0$ and $X'=Z'_0$ are smooth proper $k$-schemes. Suppose that a $K$-isomorphism
\[
\gamma_K\colon Z_K\xrightarrow{\simeq}Z'_K
\]
is specified. After replacing $C$ by a Zariski neighborhood of $0$, the conclusions of Theorem~\ref{thm:H-invariance-intro} hold for $X$ and $X'$.
\end{cor}

\begin{proof}
After replacing $C$ by an affine neighborhood of $0$, the finite-presentation limit theorem \cite[Lemma~70.7.1, Tag~07SK]{Stacks} spreads $\gamma_K$, its inverse, and the two composition identities to a common nonempty open subset $V\subset C$. Replace $C$ by
\[
C\setminus\bigl((C\setminus V)\setminus\{0\}\bigr),
\]
so the point $0$ is not removed and the punctured curve of the new base lies in $V$. The spread maps then restrict to an isomorphism $\gamma\colon Z_U\xrightarrow{\simeq}Z'_U$. The resulting data are an H-relation, so Theorem~\ref{thm:H-invariance-intro} applies.
\end{proof}

\section{Birational hyper-K\"ahler varieties}\label{sec:riess}

Rie{\ss}'s deformation families supply the H-data needed for Theorem~\ref{thm:H-invariance-intro} after the base curve is shrunk around the marked point.

\begin{prop}\label{prop:riess-H-data}
Let $X$ and $X'$ be birational projective irreducible holomorphic symplectic varieties over $\C$. Then, after replacing the curve in \cite[Proposition~2.1]{Rie14} by a Zariski neighborhood of the marked point, the two families in that proposition form an H-relation from $X$ to $X'$ in the sense of Definition~\ref{defn:H-relation}.
\end{prop}

\begin{proof}

The construction in \cite[Section~2 and Proposition~2.1]{Rie14} gives a smooth pointed curve $(C,0)$ and integral separated finite-type algebraic spaces
\[
\mathscr X\to C,\qquad \mathscr X'\to C,
\]
with special fibers $X$ and $X'$ and an isomorphism over $U=C\setminus\{0\}$. 
Here we need to slightly modify Rie{\ss}'s construction to get two smooth proper families. The properness is used when Rie{\ss} consider \textit{proper} pushforward of cycles and smoothness is used when considering relative Chern classes.

In the proof of \cite[Proposition~2.1]{Rie14}, the first family is a pullback from a projective family over a Hilbert scheme, and is therefore projective.

After shrinking $C$ around $0$, the two dominating integral families are flat over $C$: \'etale-locally this is the torsion-free criterion over the local DVRs of the regular curve $C$ \cite[Tag~0539]{Stacks}. Since the special fiber of the first projective family is the smooth variety $X$, the fiberwise criterion gives smoothness along $X$ \cite[Tag~01V9]{Stacks}. Its non-smooth locus is closed, and its image in $C$ is closed by properness of the first family, so removing that image away from $0$ makes the first family smooth. The punctured isomorphism then gives smooth fibers for the second family over $U$, while its special fiber is the smooth variety $X'$; applying the same fiberwise criterion \'etale-locally makes the second family smooth. The first family is a scheme because it is projective over the scheme $C$, and the number of geometric connected components in this proper smooth family is locally constant \cite[Tag~0E0N]{Stacks}. Since $X$ is geometrically connected, its generic fiber is geometrically connected. 

And it remains only to prove properness of $\mathscr X'\to C$. By Nagata compactification for algebraic spaces \cite[Theorem~1.2.1]{CLO12}, choose an open immersion
\[
j\colon \mathscr X'\hookrightarrow Y
\]
over $C$, with $Y$ integral and proper over $C$; here $Y$ is the reduced closure in a compactification. The whole punctured family $\mathscr X'_U$ is proper over $U$, since it is isomorphic to the proper family $\mathscr X_U$. Thus the dense open immersion $j_U\colon\mathscr X'_U\hookrightarrow Y_U$ is proper by \cite[Tag~04NX]{Stacks}, hence an isomorphism.

Let $K=\C(C)$. Then
\[
Y_K\simeq \mathscr X'_K\simeq \mathscr X_K,
\qquad
H^0(Y_K,\mathcal O_{Y_K})=K,
\]
because the generic fiber of $\mathscr X\to C$ is proper, smooth, and geometrically connected. Since $C$ is normal and integral, $Y$ is integral, and the displayed $H^0$ condition holds, \cite[Tag~0AYI]{Stacks} gives connected geometric fibers for $Y\to C$. In particular, $Y_0$ is connected.

The open immersion
\[
X'=\mathscr X'_0\hookrightarrow Y_0
\]
is proper because $X'$ is proper over $\C$ and $Y_0$ is separated over $\C$ \cite[Tag~04NX]{Stacks}. Its image is nonempty, open, and closed, hence all of $Y_0$. Together with $j_U$ being an isomorphism, this makes $j$ a surjective open immersion, hence an isomorphism. Therefore $\mathscr X'\to C$ is proper, and the two families form an H-relation.
\end{proof}

\begin{proof}[Proof of Theorem~\ref{thm:HK-intro}]
Put $k=\C$. Proposition~\ref{prop:riess-H-data} gives an H-relation from $X$ to $X'$. Theorem~\ref{thm:H-invariance-intro} gives an equivalence
\[
a_\gamma\colon p'_*\one_{X'}\xrightarrow{\simeq}p_*\one_X
\]
in $\CAlg(\SH(k))$, and its dual
\[
h_\gamma\colon \Sigma_T^\infty X_+\xrightarrow{\simeq}\Sigma_T^\infty X'_+
\]
is an equivalence of counital cocommutative coalgebra objects in $\SH(k)$.

For any $E\in\SH(k)$, precomposition with $h_\gamma$ gives an equivalence of mapping spectra
\[
\mapSH{k}{M(X')}{E}\xrightarrow{\simeq}\mapSH{k}{M(X)}{E}.
\]
If $E$ is a homotopy commutative unital ring object of $\Ho(\SH(k))$, the cup product on represented cohomology is defined using the multiplication of $E$ and the geometric diagonal of the source, and the unit and coefficient action are defined using the augmentation to $\one_k$. Since $h_\gamma$ preserves the diagonal and the augmentation, precomposition with $h_\gamma$ carries cup products, units, and the $E^{*,*}(k)$-algebra structure for $X'$ to those for $X$. Thus it induces an isomorphism of bigraded $E^{*,*}(k)$-algebras
\[
E^{*,*}(X')\xrightarrow{\simeq}E^{*,*}(X).
\]

For $E\in\CAlg(\SH(k))$, the same comparison is an equivalence of $E_\infty$-algebras over $E(k)=\mapSH{k}{\one_k}{E}$. Indeed, smooth proper duality and the projection formula identify $\uHom_k(M(X),E)$ with $B_X\tensor E$ as a commutative $E$-algebra, and similarly for $X'$. Applying the lax symmetric monoidal functor $\mapSH{k}{\one_k}{-}$ to the extension of $a_\gamma$ along $\one_k\to E$ gives the claim.

For $E=\KGL$, spectrum-level representability of homotopy invariant K-theory \cite[Sections~2.1 and~2.16, Proposition~2.19, Theorem~2.20]{Cis13} gives, for $Y=X$ or $X'$, natural equivalences
\[
\mapSH{k}{M(Y)}{\KGL}\simeq\KH(Y)\simeq \K(Y).
\]
The last equivalence holds because $Y$ is smooth, hence regular, so homotopy invariant K-theory agrees with nonconnective algebraic K-theory. These identifications can be made as equivalences of presheaves of $E_\infty$-rings on smooth $k$-schemes by \cite[Theorems~1.1 and~6.1]{NSO15}; their naturality identifies the maps from $\K(k)$. Applying the mapping-spectrum equivalence to $E=\KGL$ gives an equivalence
\[
\K(X')\xrightarrow{\simeq}\K(X)
\]
of $E_\infty$-algebras over $\K(k)$. For $n\in\Z$, we use the degree convention
\[
\K_n(Y)=\KGL^{-n,0}(Y),
\]
justified for the regular schemes $Y=X,X'$ by the preceding spectrum representation and the equivalence $\KH(Y)\simeq \K(Y)$. The multiplication agrees with the ordinary K-theory product by \cite[Corollary~3.9 in arXiv:0709.3905v2]{PPR09}, so the coefficient-algebra statement gives an isomorphism of graded K-theory rings.
\end{proof}

\urlstyle{same}

\end{document}